\documentclass[11pt]{article}
\usepackage{hyperref}
\usepackage{palatino}

\usepackage{amsmath, amsthm, amsfonts, amssymb, mathtools,nicefrac, mathrsfs}
\newtheorem{theorem}{Theorem}
\newtheorem{lemma}[theorem]{Lemma}

\newtheorem*{conjecture*}{Conjecture}

\theoremstyle{definition}
\newtheorem{definition}{Definition}[section]

\usepackage{tikz}

\usepackage{fullpage}
\usepackage{titling,soul,enumitem}  % Move the title up!
\usepackage[symbol]{footmisc}

\usepackage{algorithm, algpseudocode, subcaption, wrapfig, tikz, placeins}
\usetikzlibrary{graphs}

\DeclareMathOperator{\ucd}{\operatorname{ucd}}           % # of unique branches in the unicyclic graph.
\newcommand{\modulo}[1]{~(\mathrm{mod} \ #1)}

\DeclareMathOperator{\Deck}{\mathrm{Deck}}

\usepackage{xcolor}

\usepackage[backend=biber,style=alphabetic]{biblatex}
\usepackage[affil-it]{authblk}

\title{Reconstructing edge-deleted unicyclic graphs}

\author{Anthony E. Pizzimenti \thanks{Supported by NSF the NSF GRFP (\#2024370283).}}
\author{Umarkhon Rakhimov \thanks{Partially supported by NSF grant DMS-2408877.}}

\affil{Department of Mathematical Sciences, George Mason University}
\begin{document}
\stepcounter{footnote}
\stepcounter{footnote}
\maketitle

\begin{abstract}\noindent\sloppy
    The Harary reconstruction conjecture states that any graph with more than four edges can be uniquely reconstructed from its set of maximal edge-deleted subgraphs \cite{harary1965reconstruction}. In 1977, M\"uller verified the conjecture for graphs with $n$ vertices and $n \log_2(n)$ edges, improving on Lov\'az's bound of $\nicefrac{(n^2-n)}{4}$ \cite{Muller_1977}. Here, we show that the reconstruction conjecture holds for graphs which have exactly one cycle and three non-isomorphic subtrees.
\end{abstract}

\section{Introduction}
    A \emph{graph} $G$ is a pair $G=(V,E)$ of \emph{vertices} $V$ and \emph{edges} $E \subseteq V \times V$. A graph is \emph{simple} when $(u,v)=(v,u)$ for any edge $(u,v)$, and no edge is of the form $(u,u)$. A (length-$n$) \emph{path} in $G$ is a sequence of edges $p=(e_1, \dots, e_n)$ such that: $e_i$ and $e_j$ share a vertex if and only if $j = (i+1)\modulo n$; any vertex is shared by exactly two edges in $p$. A graph $G$ is \emph{connected} if and only if there exists a path beginning at $u$ and terminating at $v$ for any pair of distinct vertices $u\neq v$. A \emph{cycle} is a path which begins and ends at the same vertex, and a graph is \emph{unicyclic} when it has exactly one cycle, as in Figure \ref{figure:unicyclic-graph-example}; a graph without cycles is called \emph{forest}, and a connected forest graph is called a \emph{tree}. For two graphs $G$ and $H$, a bijective map $\varphi: V_G \to V_H$ where $$ (u, v) \in E_G \iff (\varphi(u), \varphi(v)) \in E_H $$ for all pairs of distinct vertices $u$ and $v$ is called a \emph{graph isomorphism}; if such a $\varphi$ exists, $G$ and $H$ are \emph{isomorphic}. From here, all graphs are finite and simple.

    In 1965, Harary posited in \cite{harary1965reconstruction} that

    \begin{conjecture*}[Harary]
        For $G = (V,E)$ a graph with $|E|>4$, let $G_i$ denote the maximal subgraph of $G$ with the $i$\textsuperscript{th} edge deleted, and $\Deck(G) \coloneqq \{G_i\}_{i=1}^m$ the set of maximal edge-deleted subgraphs of $G$. If $\Deck(G) = \Deck(H)$ for some graph $H$, then $G \cong H$.
    \end{conjecture*}

    Here, we show the conjecture holds for the class $\mathscr U$ of unicyclic graphs with at least three non-isomorphic subtrees.
    
\section{Edge reconstruction}
    \subsection{Lemmas}
        Let $G$ be a unicyclic graph and $C(G) \subseteq G$ the cycle in $G$. We can present $G$ as $C(G) \cup \mathcal T$, where $\mathcal T = \{T_1, \dots, T_p\}$ is the set of mutually disjoint subtrees of $G$, each sharing exactly one vertex with $C(G)$; the $T_k$ are the \emph{trunks} of $G$.

        A \emph{branch} $B$ of $G$ is a subgraph of a trunk $T$ such that $B$
    \begin{enumerate}[topsep=0pt, itemsep=0pt, label=(\roman*)]
        \item has at least one edge;
        \item shares exactly one vertex $v$ with $C(G)$;
        \item is connected to $v$ by only one edge;
        \item is \emph{not} a proper subgraph of any other branch $B'$.
    \end{enumerate}
    
    \begin{figure}
    \centering
    \begin{subfigure}[t]{0.49\textwidth}
        \centering
        \begin{tikzpicture}[scale=1,auto=left,every node/.style={circle,draw=black,fill=white,thick}]
          % Central cycle nodes
          \node (n0) at (0,0) {};
          \foreach \angle/\name in {0/n1, 120/n2, 240/n3}
            \node (\name) at (\angle:1cm) {};
        
          % Draw the central cycle
          \draw (n0) -- (n1) -- (n2) -- (n3) -- (n0);
        
          % Label for the graph description
          % \node[rectangle,fill=none,text width=5cm,align=center] at (0,-2.5) {Unicyclic graph with three distinct branches (red, green, blue)};
        
          % Branches off the cycle
          % Red branch forming "Y"
          \node (r1) at (3.5, 0.5) {};
          \node (r2) at (4.5, 1) {};
          \node (r3) at (4.5, 0) {};
          \draw[red, thick] (n1) -- (r1);
          \draw[red, thick] (r1) -- (r2);
          \draw[red, thick] (r1) -- (r3);
        
          % Green branch with one edge
          \node (g1) at (0.5, 1.5) {};
          \draw[green, thick] (n1) -- (g1);
        
          % Blue branch with a chain of three edges
          \node (b1) at (2, -1) {};
          \node (b2) at (3, -0.8) {};
          \node (b3) at (4.5, -1) {};
          \draw[blue, thick] (n1) -- (b1);
          \draw[blue, thick] (b1) -- (b2);
          \draw[blue, thick] (b2) -- (b3);
        
        \end{tikzpicture}
        \caption{A unicyclic graph with three distinct branches (red, green, blue). The green and blue branches are unique.}
        \label{figure:unicyclic-graph-example}
    \end{subfigure}%
    \begin{subfigure}[t]{0.49\textwidth}
        \centering
        \begin{tikzpicture}[scale=1.5,auto=left,every node/.style={circle,draw=black,fill=white,thick}]
          % Define the cycle vertices
          \node (v1) at (4,0) {};
          \node (v2) at (5,0) {};
          \node (v3) at (4.5,{sqrt(3)/2}) {};
        
          % Draw the cycle edges to form a triangle with blue color
          \draw (v1) -- (v2);
          \draw (v2) -- (v3);
          \draw (v3) -- (v1);

          % Add additional vertices connected to the cycle
          \node (v4) at (3,0) {}; % New vertex v4 to the left of v1
          \draw[green, thick] (v1) -- (v4);
        
          % Create five new vertices to the left
          \node (v5) at (6.5,0) {};
          \node (v6) at (6.5,{sqrt(3)/2}) {};
          \node (v7) at (5.5,.5) {};
          \node (v8) at (7,.5) {};
        
          % Connect the new vertices as requested
          \draw[green, thick] (v5) -- (v2);
          \draw[green, thick] (v3) -- (v6);
          \draw[green, thick] (v3) -- (v7);
          \draw  (v7) -- (v8);
          % Add additional vertices connected to the cycle
        \end{tikzpicture}
        \caption{A unicyclic graph $G$ with $\ucd(G) = 4$.}
    \end{subfigure}
    \caption{}
    \label{figure:double-figure}
\end{figure}
    
    \begin{definition}
        For a unicyclic graph $G$ define a function: $$\ucd(G) := \sum_{\mathclap{d_G(v)\geq 3,\,v\in V(C(G))}} (d_G(v)-2),$$
        which is the quantity of branches of graph $G$.
    \end{definition}
    
    \begin{lemma}\label{lemma:unicyclic-graph-categories}
        Suppose that $G=(V,E)$ is unicyclic, and set $p \coloneqq ucd(G)$. Then $G_i = (V,E \setminus \{e_i\})$ for $e_i \in E(G)$ is one of the following:
        \begin{enumerate}[itemsep=0pt, topsep=0pt]
        \item a tree;
        \item a graph with two connected components, one of which is unicyclic graph with $p$ branches, and the other a forest;
        \item graph with two connected components, one of which is unicyclic graph with $p-1$ branches, and the other a forest
        \end{enumerate}
        for any $i=1,\ldots,m$, where $m=|E|$.
    \end{lemma}
    
    \begin{proof}
    Let  $e_i$ be a edge in $G$. If $e_i \in E(C(G))$, deleting it produces a connected graph $G'$ on $n = |V(G)|$ vertices and $n-1$ edges; thus, $G'$ is a tree. If $e_i \notin V(C(G))$, there are two cases:
    
    \begin{description}[leftmargin=0.6in, style=nextline, after=\vspace{-2em}]
        \item[Case 1.] $e_i=(v,v')$, where $v' \,\in  V(G) -V(C(G)) $, $v\,\in V(C(G))$. Notice that  deleting the edge $e_i$ here increases the number of connected components by one. Otherwise there would exists two paths to vertex  $v'$ in graph $G$, consequently the path would belong to some cycle and because there is one cycle and $v'\in  V(G) -V(C(G))$, we get a contradiction. Since we did not delete any vertex or edge of $C(G)$, there exists connected subgraph  $G'$ of graph $G_i = (V,E\backslash e_i)$ such that $G'$ is unicyclic graph. Since by deleting $e_i$ we decrease the number of branches by one we have $\ucd(G')=\ucd(G)-1$. Another connected component of  $G_i$ other than $G'$ will be forest, i.e. a tree.
        \item[Case 2.] $e_i=(v,v'),\,\, v,v' \in V(G)- V(C(G))$. Deleting the edge in this case will not affect $\ucd(G)$, but it will increase the number of connected components by one.  Otherwise there would exists a vertex such that has two paths to either  $v $ or $v'$, i.e. one or both belongs two one cycle. There is the contradiction. Since we did not delete edge or vertex of the cycle  $C(G)$, then there exists subgraph $G'$ of graph $G_i = (V,E\backslash e_i)$, such that $G'$ unicycle.  As we mentioned before we did not change number of branches of $G$ by deleting the edge $e_i$, that's why $\ucd(G')=\ucd(G)$. Another connected component of  $G_i$ other than $G'$ will be forest, i.e. a tree.
    \end{description}
    \end{proof}
    
    We introduce two definitions which allow us to differentiate between branches which, up to a choice of root, may have the same structure; detailed descriptions for rooted trees and isomorphisms between them can be found in \cite{gross2018graph}.
    
    \begin{definition}[Isomorphic as rooted trees, unique branch]
        Rooted trees $T_1$ and $T_2$ with roots $v_1,\,v_2$ are \emph{isomorphic as rooted trees} if there is a graph isomorphism between $T_1$ and $T_2$ which maps $v_1$ to $v_2$. A branch $B$ of a unicyclic graph $G$ is called \textit{unique} when, for any other branch $B'$ disjoint from $B$, $B$ and $B'$ are not isomorphic as rooted trees.
    \end{definition}
    
    \begin{lemma} \label{lemma:identifiable-branches}
        Any unique branch of the unicyclic graph $G=(V,E)$, where $C(G)$ has at least length five and $\ucd(G)\geq 5$, can be reconstructed from the $\Deck(G)$.
    \end{lemma}
    
    \begin{proof}
        Via Lemma \ref{lemma:unicyclic-graph-categories}, some $G'\in \{G_i\}^m_{i=1}$ will have two components, one of which is unicyclic component $U$ with $\ucd(G)-1$  branches.  Denote the collection of  these unicyclic components for all $G' \in \{G_i\}^m_{i=1}$ by $\mathcal{U}$ like in Figure \ref{figure:deck}. From $\mathcal U$, choose any three graphs $U_1, U_2,$ and $U_3$. By the pigeonhole principle and the construction of $\mathcal U$, at least two of the $U_i$ contain the unique branch $B$. Without loss of generality, suppose that $U_1$ and $U_2$ are such that $B$ is a branch of both of them; let $\cong$ be an equivalence relation on the branches of $G$ such that $B_1 \cong B_2$ if and only if $B_1$ and $B_2$ are isomorphic as rooted trees. Then, let $\varphi_1$ and $\varphi_2$ be the partitions admitted by $\cong$ on $U_1$ and $U_2$, respectively. Because both $U_1$ and $U_2$ contain $B$ and $B$ is not isomorphic to any other branch in $G$, it is the sole member of its equivalence class; as such, $\varphi = \varphi_1 \cap \varphi_2$ must contain $\{B\}$.
    \end{proof}

    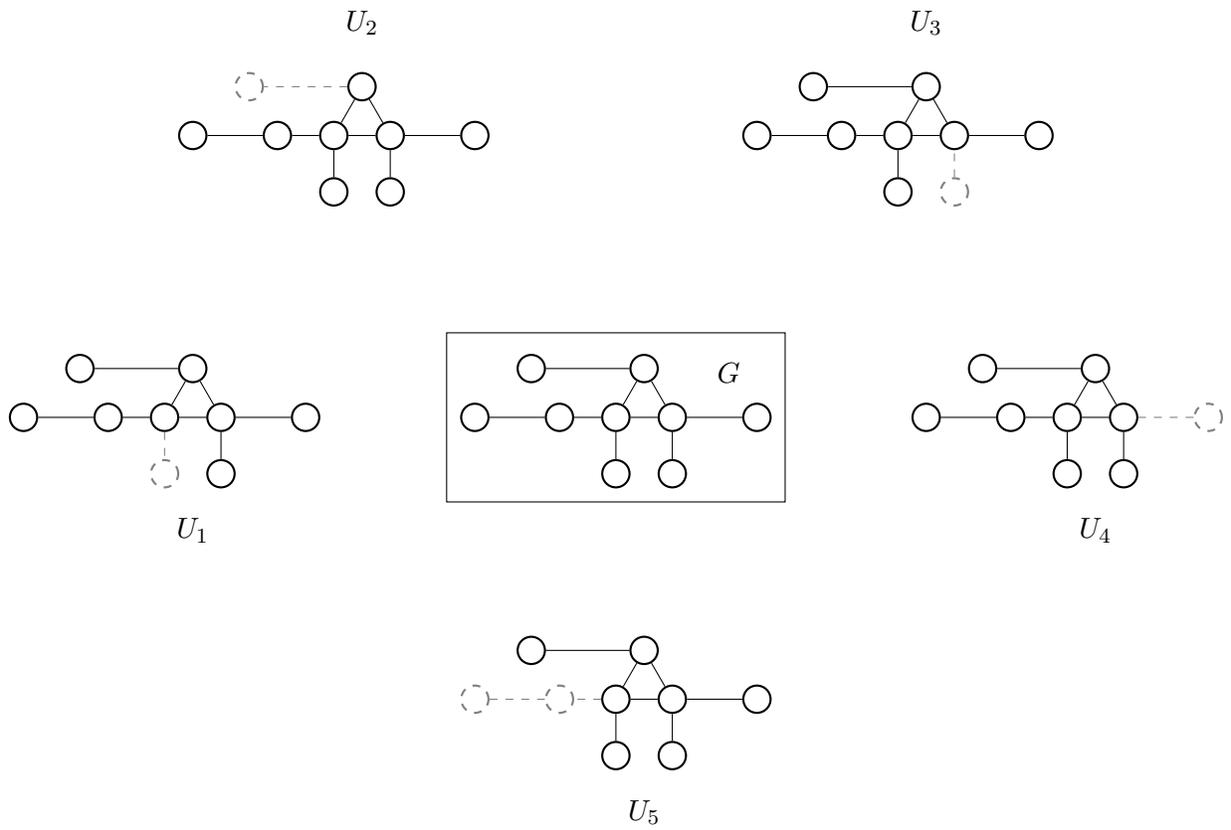
\begin{figure}
    \centering
    \begin{tikzpicture}[scale=0.75,every node/.style={circle,draw=black,fill=white,thick}]
        \begin{scope}[shift={(0,0)}]
            \node (v1) at (4,0) {};
              \node (v2) at (5,0) {};
              \node (v3) at (4.5,{sqrt(3)/2}) {};
            
              % Draw the cycle edges to form a triangle with blue color
              \draw (v1) -- (v2);
              \draw(v2) -- (v3);
              \draw (v3) -- (v1);
            
              % Add additional vertices connected to the cycle
              \node (v4) at (4,-1) {}; % New vertex v4 to the left of v1
              \draw (v1) -- (v4);
            
              % Create five new vertices to the left
              \node (v5) at (5,-1) {};
              \node (v6) at (6.5,0) {};
              \node (v7) at (1.5,0) {};
              \node (v8) at (3,0) {};
              \node (v9) at (2.5,{sqrt(3)/2}) {};
            
              % Connect the new vertices as requested
              \draw (v5) -- (v2);
              \draw (v6) -- (v2);
              \draw (v3) -- (v9);
              \draw (v1) -- (v8);
              \draw  (v7) -- (v8);
              % Add additional vertices connected to the cycle
        \end{scope}

        \node[draw=none, fill=none] at (-3.5, -2) {$U_1$};
        \begin{scope}[shift={(-8, 0)}]
            % Define the cycle vertices
            \node (v1) at (4,0) {};
            \node (v2) at (5,0) {};
            \node (v3) at (4.5,{sqrt(3)/2}) {};
            
            % Draw the cycle edges to form a triangle with blue color
            \draw (v1) -- (v2);
            \draw (v2) -- (v3);
            \draw (v3) -- (v1);
            \node[draw=gray, dashed] (v4) at (4,-1) {}; % New vertex v4 to the left of v1
            \draw[gray, dashed] (v1) -- (v4);
            
            % Create five new vertices to the left
            \node (v5) at (5,-1) {};
            \node (v6) at (6.5,0) {};
            \node (v7) at (1.5,0) {};
            \node (v8) at (3,0) {};
            \node (v9) at (2.5,{sqrt(3)/2}) {};
            
            % Connect the new vertices as requested
            \draw (v5) -- (v2);
            \draw (v6) -- (v2);
            \draw (v3) -- (v9);
            \draw (v1) -- (v8);
            \draw (v7) -- (v8);
        \end{scope}

        \node[draw=none, fill=none] at (12.5, -2) {$U_4$};
        \begin{scope}[shift={(8,0)}]
            \node (v1) at (4,0) {};
              \node (v2) at (5,0) {};
              \node (v3) at (4.5,{sqrt(3)/2}) {};
            
              % Draw the cycle edges to form a triangle with blue color
              \draw (v1) -- (v2);
              \draw (v2) -- (v3);
              \draw (v3) -- (v1);
            
              % Add additional vertices connected to the cycle
              \node (v4) at (4,-1) {}; % New vertex v4 to the left of v1
              \draw (v1) -- (v4);
            
              % Create five new vertices to the left
              \node (v5) at (5,-1) {};
              \node[draw=gray, dashed] (v6) at (6.5,0) {};
              \node (v7) at (1.5,0) {};
              \node (v8) at (3,0) {};
              \node (v9) at (2.5,{sqrt(3)/2}) {};
            
              % Connect the new vertices as requested
              \draw (v5) -- (v2);
              \draw[gray, dashed] (v6) -- (v2);
              \draw (v3) -- (v9);
              \draw (v1) -- (v8);
              \draw (v7) -- (v8);
              % Add additional vertices connected to the cycle
        \end{scope}

        \node[draw=none, fill=none] at (-0.5, 7) {$U_2$};
        \begin{scope}[shift={(-5,5)}]
            \node (v1) at (4,0) {};
              \node (v2) at (5,0) {};
              \node (v3) at (4.5,{sqrt(3)/2}) {};
            
              % Draw the cycle edges to form a triangle with blue color
              \draw (v1) -- (v2);
              \draw(v2) -- (v3);
              \draw (v3) -- (v1);
            
              % Add additional vertices connected to the cycle
              \node (v4) at (4,-1) {}; % New vertex v4 to the left of v1
              \draw (v1) -- (v4);
            
              % Create five new vertices to the left
              \node (v5) at (5,-1) {};
              \node (v6) at (6.5,0) {};
              \node (v7) at (1.5,0) {};
              \node (v8) at (3,0) {};
              \node[draw=gray, dashed] (v9) at (2.5,{sqrt(3)/2}) {};
            
              % Connect the new vertices as requested
              \draw (v5) -- (v2);
              \draw (v6) -- (v2);
              \draw[gray, dashed] (v3) -- (v9);
              \draw (v1) -- (v8);
              \draw  (v7) -- (v8);
              % Add additional vertices connected to the cycle
        \end{scope}

        \node[draw=none, fill=none] at (9.5, 7) {$U_3$};
        \begin{scope}[shift={(5,5)}]
            \node (v1) at (4,0) {};
              \node (v2) at (5,0) {};
              \node (v3) at (4.5,{sqrt(3)/2}) {};
            
              % Draw the cycle edges to form a triangle with blue color
              \draw (v1) -- (v2);
              \draw(v2) -- (v3);
              \draw (v3) -- (v1);
            
              % Add additional vertices connected to the cycle
              \node (v4) at (4,-1) {}; % New vertex v4 to the left of v1
              \draw (v1) -- (v4);
            
              % Create five new vertices to the left
              \node[draw=gray, dashed] (v5) at (5,-1) {};
              \node (v6) at (6.5,0) {};
              \node (v7) at (1.5,0) {};
              \node (v8) at (3,0) {};
              \node (v9) at (2.5,{sqrt(3)/2}) {};
            
              % Connect the new vertices as requested
              \draw[gray, dashed] (v5) -- (v2);
              \draw (v6) -- (v2);
              \draw (v3) -- (v9);
              \draw (v1) -- (v8);
              \draw (v7) -- (v8);
              % Add additional vertices connected to the cycle
        \end{scope}

        \node[draw=none, fill=none] at (4.5, -7) {$U_5$};
        \begin{scope}[shift={(0,-5)}]
            \node (v1) at (4,0) {};
              \node (v2) at (5,0) {};
              \node (v3) at (4.5,{sqrt(3)/2}) {};
            
              % Draw the cycle edges to form a triangle with blue color
              \draw (v1) -- (v2);
              \draw(v2) -- (v3);
              \draw (v3) -- (v1);
            
              % Add additional vertices connected to the cycle
              \node (v4) at (4,-1) {}; % New vertex v4 to the left of v1
              \draw (v1) -- (v4);
            
              % Create five new vertices to the left
              \node (v5) at (5,-1) {};
              \node (v6) at (6.5,0) {};
              \node[draw=gray, dashed] (v7) at (1.5,0) {};
              \node[draw=gray, dashed] (v8) at (3,0) {};
              \node (v9) at (2.5,{sqrt(3)/2}) {};
            
              % Connect the new vertices as requested
              \draw (v5) -- (v2);
              \draw (v6) -- (v2);
              \draw (v3) -- (v9);
              \draw[gray, dashed] (v1) -- (v8);
              \draw[gray, dashed] (v7) -- (v8);
              % Add additional vertices connected to the cycle
        \end{scope}

        \draw (1, -1.5) rectangle (7, 1.5);
        \node[draw=none, fill=none] at (6, 0.8) {$G$};
    \end{tikzpicture}
    \caption{The unicyclic, one-edge-deleted subgraphs $\mathcal U = \{U_1, \dots, U_5\}$ of $G$.}
    \label{figure:deck}
\end{figure}
    \FloatBarrier

\subsection{Theorem}

    \begin{theorem}
        Suppose that $G$ is unicyclic and
        \begin{enumerate}[topsep=0pt, itemsep=0pt]
            \item $ \ucd(G)\geq5$;
            \item there are at least three unique branches with (pairwise) distinct roots.
        \end{enumerate}
        Then, $G$ is edge-reconstructable.
    \end{theorem}
    
    \begin{proof}
        By Lemma \ref{lemma:identifiable-branches}, we can identify the unique branches $B_1, B_2, B_3$ of $G$; consider the set $\mathcal U$ of maximal unicyclic subgraphs of $G$ from Lemma \ref{lemma:identifiable-branches}. Because $\ucd(G)$ counts the number of branches, $\mathcal U$ is a collection of $\ucd(G)$ elements, each a unicyclic subgraph of $G$ obtained by deleting exactly one branch of $G$. Then, $\mathcal U$ contains at least five subgraphs $U_1, \dots, U_5$. By the pigeonhole principle, at least two of the $U_i$ have $B_1, B_2, B_3$ as subgraphs; without loss of generality, let $U_1$ and $U_2$ contain $B_1, B_2,$ and $B_3$. There are two branches $\Gamma_1, \Gamma_2$ of $G$ such that $U_1$ contains $\Gamma_1$ but not $\Gamma_2$ and $U_2$ contains $\Gamma_2$ but not $\Gamma_1$, implying that $U_1 \cup U_2 = G$.
    
        Let $x_i$ be the root of $B_i$ for $i \in \{1,2,3\}$. Then, for $i \neq j$ there exist two paths from $x_i$ and $x_j$ on $C(G)$. As there are three roots, there is at least one pair of roots (without loss of generality, $x_1$ and $x_2$) such that one of these paths is longer than the other.
    
        \begin{algorithm}[hbt!]
        \begin{algorithmic}
            \State $\widetilde{G} \coloneqq B_1$
            \vspace{1em}

            \For{each $i$ in $\{1, \dots, |S|-1\}$}
                \State $E_{\widetilde{G}} \coloneqq  E_{\widetilde{G}} \cup (S_i, S_{i+1}) $ \Comment{{\footnotesize Add the $i$\textsuperscript{th} edge to our partial reconstruction $\widetilde G$.}}
                \State $\widetilde{G} \coloneqq \widetilde{G} \cup A^S_i$
            \EndFor
            \vspace{1em}
            \For{each $j$ in $\{1, \dots, |L|-1\}$}
                \State $E_{\widetilde{G}} \coloneqq E_{\widetilde{G}} \cup (L_j, L_{j+1}) $
                \State $\widetilde{G} \coloneqq \widetilde{G} \cup A^L_j$
            \EndFor
        \end{algorithmic}
        \caption{Let $S = (x_1, \dots, x_2)$ and $L = (x_1, \dots, x_2)$ be the sequences of vertices which define the respective shortest and longest paths on $C(G)$ from $x_1$ to $x_2$. For any vertex $S_i$ (resp. $L_j$) let $A^S_i$ (resp. $A^L_j$) be the union of all branches of $G$ which contain $S_i$ --- that is, $A^S_i$ is the union of all branches of $S_i$ in $U_1$ and $U_2$.}
    \end{algorithm}
    
        The above algorithm constructs $\widetilde G = U_1 \cup U_2$ which, by the previous paragraph, is $G$.
        \end{proof}

    \section*{Acknowledgements}
        The authors express their thanks to Drs. Geir Agnarsson and Harbir Antil for stimulating conversations and robust critiques.
    
    \newpage
    \printbibliography
\end{document}